\documentclass[12pt]{article}

\usepackage{amsmath,enumerate,amssymb,amsthm}

\newtheorem{theorem}{Theorem}[section]
\newtheorem{lemma}{Lemma}[section]
\newtheorem{corollary}{Corollary}[section]

\theoremstyle{remark}
\newtheorem{rmk}{Remark}
\newtheorem*{example}{Example: the Wilson matrix}
\theoremstyle{definition}
\newtheorem{defn}{Definition}

 \font\smallit=cmti10
 \font\smalltt=cmtt10

\newcommand{\be}{\begin{equation}}
\newcommand{\ee}{\end{equation}}
\newcommand{\ben}{\begin{enumerate}}
\newcommand{\een}{\end{enumerate}}

\def\emx{{\cal E}}

\def\imx{{\cal I}}
\def\omx{{\cal O}}

\def\wt{\mathop{\rm wt}}
\def\rank{\mathop{\rm rk}}
\def\rank{\mathop{\mathrm{rank}}}
\def\adj{\mathop{\rm adj}}
 
 \long\def\symbolfootnote[#1]#2{\begingroup%
\def\thefootnote{\fnsymbol{footnote}}\footnote[#1]{#2}\endgroup}

\begin{document}
\renewcommand\refname{\normalsize References}

\begin{center}
 {\bf INTEGER MATRIX FACTORISATIONS, SUPERALGEBRAS AND THE QUADRATIC FORM OBSTRUCTION}
 \vskip 20pt
 {{\bf Nicholas J. Higham$^a$}, {\bf Matthew C. Lettington$^b$} and {\bf Karl Michael Schmidt$^b$}}\\
  {\smallit $^a$Department of Mathematics, University of Manchester, Manchester, M13 9PL, UK}\\
 {\smallit $^b$School of Mathematics, Cardiff University, 23 Senghennydd Road, Cardiff, UK CF24 4AG}\\
 {\smalltt nick.higham@manchester.ac.uk: LettingtonMC@cardiff.ac.uk: SchmidtKM@cardiff.ac.uk}\\
 \end{center}
 \vskip 30pt

\begin{abstract}

\noindent
{We identify and analyse obstructions to factorisation of integer matrices
into products $N^T N$ or $N^2$ of matrices with rational or integer
entries.  The obstructions arise as quadratic forms with integer
coefficients and raise the question of the discrete range of such
forms. They are obtained by considering matrix decompositions over a
superalgebra.  We further obtain a formula for the determinant of a square
matrix in terms of adjugates of these matrix decompositions, as well as
identifying a \emph{co-Latin} symmetry space.}

\end{abstract}


 \baselineskip=15pt

\def\rc#1{\frac{1}{#1}}
\def\cases#1{\left\{\begin{matrix} #1 \end{matrix}\right.}
\def\pmatrix#1{\left(\begin{matrix} #1 \end{matrix}\right)}
\def\nonpmatrix#1{\begin{matrix} #1 \end{matrix}}


\section{\large Introduction}
\label{intro}
The question whether a given square integer matrix $M$ can be factorised
into a product of two
integer matrices, either in the form of a
square $M=N^2$ or (in case of a symmetric positive definite matrix)
in the form $M=N^T N$,
has a long history in number theory. It is known that if $n\leq 7$ and $M$
is an $n\times n$ symmetric positive definite matrix with integer entries
and determinant 1, then a factorisation $M=N^T N$ with $N$ an $n\times n$
matrix with integer entries exists. However, there are examples of such
matrices with dimension $n=8$ which cannot be factorised in this way.  This
result is mentioned by Taussky (see \cite[p.~812]{taussky1}) and goes
back to Hermite, Minkowski, and Mordell Mordell~\cite{mordell1}.

\symbolfootnote[0]{
2010 \emph{Mathematics Subject Classification}: 15A23, 15CA30, 11H55, 05B15.\newline
\emph{Key words and phrases}: integer matrix factorisation, matrix superalgebra, quadratic forms, latin squares, adjugate matrix.}

The number theoretic properties relating to the factorisation of symmetric
positive definite $n\times n$ integer matrices $M$ with fixed determinant
have classical connections to the theory of positive definite quadratic
forms in $n$ variables, see e.g. \cite{mordell2} and the above references.

In particular, Mordell considered the similarity classes of $n\times n$
matrices with determinant 1, where two such matrices $L$, $M$ are in the
same class if there exists a unimodular integral matrix $N$ such that
$M= N^T LN$. The number of such similarity classes is denoted by
$h_n$. Then a matrix $M$ is in the similarity class of $I_n$ (the identity
matrix) if and only if there exists a factorisation $M= N^T N$ with an
integer matrix $N$. This implies that the quadratic form classically
associated with the symmetric matrix $M$, $q(x)= x^T M x$, can be written
as \be q(x)=x^T M x
    = x^TN^T Nx
    = y^Ty = \sum_{j=1}^n y_j^2,
\label{eq:qtemp}
\ee
where $y= N x$, and $N$ has determinant 1. Thus, the factorisation can be
used to write the quadratic form $q(x)$ as a sum of squares of $n$ linear
factors.  When $n=8$, such a factorisation may not exist, as Minkowski
proved in 1911 that $h_n\geq [1+n/8]$, so $h_n\geq 2$ if $n=8$. Mordell
showed that $h_8=2$ \cite{mordell2}, and Ko showed that $h_9 = 2$ as well \cite{ko1}.

In the present paper, we revisit the question of integer matrix
factorisation in the light of recent general results on matrix
decompositions \cite{slh1}, \cite{slh2}.
We establish in Corollary~\ref{thm:23}
that the existence of integer
solutions to a certain quadratic equation
is a necessary condition for a matrix
factorisation of the type $M=N^2$ or $M=N^TN$ (for symmetric positive
definite $M$) to exist.  It is interesting to
note that solutions to this new type of quadratic equation associated with a
given integer matrix $M$ can also lead to rational matrix factors $N$ with
entries in $\frac{1}{n^2}\mathbb{Z}$.

Throughout the paper, we use the classical example of the
Wilson matrix \cite{higham1}, \cite{higham2}, \cite{moler1}, \cite{morr46}
\be
W = \left ( \begin{array}{cccc}
5 &7& 6& 5 \\
 7 &10& 8& 7 \\
 6& 8& 10& 9 \\
5& 7& 9 &10 \\
\end{array}\right )
\label{eq:wilson}
\ee
to demonstrate the methodologies under consideration.
This integer matrix has determinant 1 and hence an integer inverse matrix, but is moderately ill conditioned, despite its small size and entries.
It has the integer factorisation $W = Z^T Z$ discovered in
\cite{higham2} with
\be
Z = \left ( \begin{array}{cccc}
2 & 3 & 2 & 2 \\
1 & 1 & 2 & 1 \\
0 & 0 & 1 & 2 \\
0 & 0 & 1 & 1
\end{array}\right ).
\label{eq:wilroot}
\ee
We note that the entries of $Z$ are nonnegative and, although the matrix is not triangular, it has a block upper triangular structure and can be thought of as a block Cholesky factor of $W$.

The quadratic form associated with the Wilson matrix can be written,
using (\ref{eq:qtemp}) and \eqref{eq:wilroot}, as a sum of four squares:
\begin{align}
\label{eq:qform0}
q(x) = x^T W x &= (x_3+x_4)^2+(x_3+2 x_4)^2+(x_1+x_2+2 x_3+x_4)^2  \nonumber \\
     & \quad {} +(2 x_1+3 x_2+2 x_3+2 x_4)^2.
\end{align}
As $Z$ is a unimodular integer matrix and hence has
an integer inverse, it follows by Lagrange's four-square theorem that the
quadratic form $q$ generated by the Wilson matrix is universal
\cite{conway99} in the sense that it generates all positive integers as $x$
ranges over $\mathbb{Z}^4$.

This shows that integer matrix factorisation is a valuable tool in studying
the quadratic form generated by an integer matrix. As a result of our
considerations in Section 3 below, the question of factorising the Wilson
matrix in the form $W=Z^TZ$ is associated with the solutions of the
quadratic equation
\begin{equation}\label{eq:qform1}
2 w^2+x_1^2+x_1 x_2+x_1 x_3+x_2^2+x_2 x_3+x_3^2=952.
\end{equation}
Indeed, a necessary (but not sufficient) condition for $W$ to factorise is
that integer solutions $(w,x_1,x_2,x_3)$ to the quadratic equation
(\ref{eq:qform1}) exist.
Thus this equation can be
considered a quadratic form obstruction to integer factorisability.  This
approach to integer matrix factorisation was briefly alluded to, but not
fully worked out in \cite{higham2}.

In Section~\ref{sec.sv-decomp-matr} we derive a useful
$\mathrm{S}+\mathrm{V}$ decomposition of square matrices, first identified in \cite{slh2}, and
give explicit formulas for constructing it
that were not given in \cite{slh2}.
In Section~\ref{sec:quad}, we use the
concept of matrix weight, associated with the S factor,
to establish the quadratic form obstruction to
integer matrix factorisation and show how a solution of the corresponding
quadratic equation can be used to calculate the matrix factors.  In
Section~\ref{sec.determ-decomp-adjug}, we discuss adjugate matrices in view
of the matrix decomposition and in particular show that the type S part of
a matrix is characterised by having an adjugate with all equal
entries. Finally, in Section~\ref{sec:colatin} we identify the type V part
of a matrix as belonging to a space of co-Latin squares, defined as square
matrices with constant sum over all entries carrying the same symbol in any
Latin square.

\section{\large The $\mathrm{S}+\mathrm{V}$ decomposition of matrices}\label{sec.sv-decomp-matr}

The following symmetries of ${n\times n}$ matrices were considered in
\cite{slh2}.

(S) A matrix $M = (m_{i,j})_{i,j=1}^n\in{\mathbb R}^{n\times n}$ has the
{\it constant sum property\/} (or is of type S) if there is a number
$w\in {\mathbb R}$, called the {\it weight\/} of the matrix, such that
\[
\sum_{j=1}^n m_{i,j}=\sum_{j=1}^n m_{j,i}=nw \qquad (i \in \{1, \dots, n\}).
\]
The vector subspace of ${\mathbb R}^{n \times n}$ of matrices having the constant sum
property with some weight is denoted by $S_n$ and can be characterised as
\[
S_n = \bigl\{M \in {\,\mathbb R}^{n \times n} : 1_n^T M u = 0 = u^T M 1_n\; (u \in
                 \{1_n\}^\bot)\,\bigr\},
\]
where $1_n \in {\mathbb R}^n$ is the column vector with all entries equal
to 1 and orthogonality is with respect to the standard inner product,
$\{1_n\}^\bot = \{u \in {\mathbb R}^n : u^T 1_n = 0\}$ (cf.\
\cite[Thm.~2.6~(a)]{slh2}).

(V) A matrix $M = (m_{i,j})_{i,j=1}^n\in{\mathbb R}^{n\times n}$ has the
{\it vertex cross sum property\/} (or is of type V) if
\[
m_{i,j}+m_{k,l}=m_{i,l}+m_{k,j} \quad (i,j,k,l \in \{1, \dots, n\})
\]
and the matrix entries sum to zero, $\sum_{i,j=1}^n m_{i,j} = 0$.  The
vector subspace of ${\mathbb R}^{n \times n}$ of matrices having the vertex
cross sum property is denoted by $V_n$ and can be characterised as
\begin{equation}
 V_n = \bigl\{\,M \in {\,\mathbb R}^{n \times n} : u^T M v = 0\; (u, v\in \{1_n\}^\bot), \ 1_n^T M 1_n = 0\,\bigr\}
\label{eq:vprop}
\end{equation}
(cf.\ \cite[Thm.~2.6~(e)]{slh2}).
(We derive a surprising alternative characterisation of this space in Section \ref{sec:colatin}.)

The spaces $S_n$ and $V_n$ only have the null matrix in common; in fact,
they complement each other and give ${\mathbb R}^{n\times n}$ a
superalgebra structure in the following way
(cf.\ \cite[Thm.~2.5~(a)]{slh2}).

\begin{theorem}
\label{thm:sv}
${\mathbb R}^{n\times n} = S_n \oplus V_n$.
Moreover, $S_n$ is a subalgebra of ${\mathbb R}^{n\times n}$, and
\[
 S_n S_n \subset S_n, \quad S_n V_n \subset V_n, \quad V_n S_n \subset V_n, \quad V_n V_n \subset S_n.
 \]
\end{theorem}

We show below in Corollary \ref{cor:21} that this decomposition of the
matrix algebra is orthogonal.

We begin by showing in Theorem \ref{thm:21} that every element $M$ of $V_n$
can be written in the form $M=a1_n^T+1_n b^T$, before deriving a formula
for obtaining these vectors $a$ and $b$ in Theorem \ref{thm:215}.

\begin{theorem}
\label{thm:21}
A matrix $M \in {\mathbb R}^{n \times n}$ is an element of $V_n$ if and only if there exist vectors
$a, b \in \{1_n\}^\bot$ such that $M = a 1_n^T + 1_n b^T$.
\end{theorem}

\begin{proof}
Let $M \in V_n.$ Then for any $v \in \{1_n\}^\bot,$ we have
$M v \in \{1_n\}^{\bot \bot} = {\mathbb R}\,1_n;$ hence there is a linear form
$f : \{1_n\}^\bot \rightarrow{\mathbb R}$ such that $M v = 1_n\,f(v)$
$(v \in \{1_n\}^\bot).$
By the Riesz Representation Theorem, there is a vector $b \in \{1_n\}^\bot$ such that
$f(v) = b^T v$ $(v \in \{1_n\}^\bot),$ and consequently $M v = 1_n b^T v$ $(v \in \{1, n\}^\bot).$
Now every $x \in {\mathbb R}^n$ can be written in the form $x = \alpha 1_n + v$ with suitable
$v \in \{1, n\}^\bot$ and $\alpha \in {\mathbb R},$ so
\begin{align*}
 M x &= \alpha M 1_n + M v = \alpha b 1_n^T 1_n + 1_n b^T v = (a 1_n^T + 1_n b^T)(\alpha 1_n + v)
\\
 &= (a 1_n^T + 1_n b^T) x,
\end{align*}
where $a := \frac{1}{n}\,M 1_n,$ bearing in mind that $1_n^T 1_n = n.$ It
follows from the last equation in (\ref{eq:vprop}) that $1_n^T a = 0.$
Conversely, for any $a, b \in \{1_n\}^\bot$ the matrix
$a 1_n^T + 1_n b^T \in V_n$.
\end{proof}

Clearly $M\in V_n$ is symmetric if and only if $a=b$ in the above
representation.  Theorem \ref{thm:21} shows that $\dim V_n = 2n - 2$, so
$\dim S_n = n^2 - 2n + 2$ (see also \cite[p.~14]{slh2}), and that the rank
of elements of $V_n$ cannot exceed 2.
Moreover, this theorem makes the spectral decomposition of any matrix $M\in V_n$
very transparent. Indeed, the range of $M$ is spanned by the orthogonal vectors $a$ and $1_n,$
so any eigenvector for non-zero eigenvalue must be of the form $u = \alpha a + \beta 1_n$
with numbers $\alpha, \beta \in {\mathbb C}$.
Then, bearing in mind that $1_n^T a = 0$ and $b^T 1_n = 0$, and denoting the eigenvalue by $\lambda$, we find
\begin{align*}
 \lambda u = M u &= (a 1_n^T + 1_n b^T) (\alpha a + \beta 1_n)
 = n \beta a + \alpha b^T a 1_n,
\end{align*}
giving $n \beta = \lambda \alpha$ and $b^T a\,\alpha = \lambda \beta$.
Hence any nonzero eigenvalue
of $M$ is an eigenvalue of the $2 \times 2$ matrix
$\left (\begin{array}{cc}
0 & n \\
 b^T a & 0
\end{array}\right )$
and vice versa.
This gives the characteristic polynomial for
$M,$
\[
 \chi_M(\lambda) = \lambda^{n-2} (\lambda^2 - n b^T a)
\]
and furthermore the eigenvectors for
$M$ as $v_\pm = \sqrt n a \pm \sqrt{b^T a} 1_n$ for eigenvalues
$\lambda_\pm = \pm \sqrt{n b^T a}$ if $b^T a > 0,$
$v_\pm = \sqrt n a \pm i \sqrt{- b^T a} 1_n$ for eigenvalues
$\lambda_\pm = \pm i \sqrt{-n b^T a}$ if $b^T a < 0$
and the single eigenvector $a$ for eigenvalue 0 if $b^T a = 0$
(and $a \neq 0$; otherwise any non-null vector orthogonal to $b$ will do).

Finally, Theorem \ref{thm:21} easily yields the orthogonality with respect to the Frobenius inner product of the direct sum decomposition
of the space of $n\times n$ matrices into $S_n$ and $V_n$.
The Frobenius inner product of two matrices
$A, B\in {\mathbb R}^{n \times n}$ is defined as
$\left< A, B \right> = \mathop{\rm tr} A^T B$.

\begin{corollary}
\label{cor:21}
The decomposition
${\mathbb R}^{n \times n} = S_n \oplus V_n$
is orthogonal with respect to the Frobenius inner product.
\end{corollary}

\begin{proof}
Let $S \in S_n,$ $V \in V_n;$
then by Theorem \ref{thm:21} there are vectors $a, b \in \{1_n\}^\bot$ such that $V = a 1_n^T + 1_n b^T$.
Hence
\begin{align*}
 \left< S, V \right> &= \mathop{\rm tr} S^T V = \mathop{\rm tr} (S^T a 1_n^T + S^T 1_n b^T)
 = 1_n^T S^T a + b^T S^T 1_n
\\
&= a^T S 1_n + b^T S^T 1_n = 0,
\end{align*}
since $S 1_n = S^T 1_n = n w_S 1_n$,
where $w_S$ is the weight of the matrix $S$.
\end{proof}

We can calculate the weight of any matrix in $S_n$ as the mean of all matrix
entries. In fact, it is meaningful to define the linear form
\[
\mathop{\rm wt} : {\mathbb R}^{n\times n} \rightarrow {\mathbb R},
\quad
\mathop{\rm wt} M = \frac 1 {n^2} \sum_{i,j=1}^n m_{i,j}
\]
giving the weight of any $n\times n$ matrix $M = (m_{i,j})_{i,j=1}^n$.
Theorem \ref{thm:21} immediately shows that $\mathop{\rm wt} M = 0$ for all $M\in V_n$.

Any $M\in S_n$ can be uniquely decomposed into
\[
 M = M_0 + (\mathop{\rm wt} S)\,\emx_n,
\]
where $M_0 \in S_n$ has weight 0 and
$\emx_n = 1_n 1_n^T \in S_n$ is the $n\times n$ matrix with all entries equal to 1.
In conjunction with the previous observations, this gives rise to the following unique
decomposition of general $n\times n$ matrices, including an explicit formula for the calculation of the parts.

\begin{theorem}
\label{thm:215}
Let $M=(m_{i,j})\in\mathbb{R}^{n\times n}$. Then there is a unique decomposition
\[
 M = M_V + M_0 + (\mathop{\rm wt} M)\,\emx_n,
\]
where $M_0 \in S_n$ with weight 0 and $M_V=a1_n^T+1_nb^T \in V_n$, and the
entries of the vectors $a$ and $b$ are given by
\begin{align*}
a_i &= \frac{1}{n} \sum_{j=1}^n m_{i,j} - \mathop{\rm wt} M \quad(i \in\{1, \dots, n\}),\\
b_j &= \frac{1}{n} \sum_{i=1}^n m_{i,j} - \mathop{\rm wt} M \quad(j \in\{1, \dots, n\}).
\end{align*}
In particular, if $M$ is an integer matrix then the vectors $n^2a$ and
$n^2b$ have integer entries.
\end{theorem}

\begin{proof}
Let $\{\varepsilon_j : j \in \{1, \dots, n-1\}\}$ be an orthonormal basis of $\{1_n\}^\bot.$ Then
\[
\left\{\frac{1}{\sqrt n} \varepsilon_j 1_n^T, \frac{1}{\sqrt n} 1_n \varepsilon_j^T : j \in \{1, \dots, n-1\}\right\}
\]
is an orthonormal basis, with respect to the Frobenius inner product, of the ($2n-2$-dimensional) space $V_n.$ Indeed, for $j, k \in \{1, \dots,~n~-~1\}$ we have
\[
 \left<\varepsilon_j 1_n^T, 1_n \varepsilon_k^T \right> = \mathop{\rm tr} 1_n \varepsilon_j^T 1_n \varepsilon_k^T = 0
\]
since $\varepsilon_j^T 1_n = 0,$ and
\[
 \left<\varepsilon_j 1_n^T, \varepsilon_k 1_n^T \right> = \mathop{\rm tr} 1_n \varepsilon_j^T \varepsilon_k 1_n^T = \delta_{j\, k}\,\mathop{\rm tr} 1_n 1_n^T  = n \delta_{j\, k}
\]
and analogously for
$\left<1_n \varepsilon_j^T, 1_n \varepsilon_k^T \right>.$

These observations enable us to find $M_V$, the $V_n$ part of $M \in {\mathbb R}^{n \times n}$
by orthogonal projection using the above orthonormal basis. We have
\begin{align*}
M_V &= \sum_{j=1}^{n-1} \left(\left<\frac{1}{\sqrt n} \varepsilon_j 1_n^T, M \right> \frac{1}{\sqrt n} \varepsilon_j 1_n^T + \left<\frac{1}{\sqrt n} 1_n \varepsilon_j^T, M \right> \frac{1}{\sqrt n} 1_n \varepsilon_j^T \right)
\\
&= \frac{1}{n} \sum_{j=1}^{n-1} \left(\mathop{\rm tr}(1_n \varepsilon_j^T M) \varepsilon_j 1_n^T + \mathop{\rm tr}(\varepsilon_j 1_n^T M) 1_n \varepsilon_j^T \right)
\\
&= \frac{1}{n} \sum_{j=1}^{n-1} (\varepsilon_j^T M 1_n) \varepsilon_j 1_n^T + \frac{1}{n} \sum_{j=1}^{n-1} 1_n ((1_n^T M \varepsilon_j) \varepsilon_j)^T  = a_M 1_n^T + 1_m b_M^T
\end{align*}
with
\[
 a_M = \frac{1}{n} \sum_{j=1}^{n-1} (\varepsilon_j^T M 1_n) \varepsilon_j, \qquad
 b_M = \frac{1}{n} \sum_{j=1}^{n-1} (1_n^T M \varepsilon_j) \varepsilon_j.
\]
Observing that
$\{\varepsilon_j : j \in \{1, \dots, n-1\}\} \cup \{\frac{1}{\sqrt n} 1_n\}$
is an orthonormal basis of ${\mathbb R}^n,$ we find
\[
 \sum_{j=1}^{n-1} (\varepsilon_j^T v) \varepsilon_j = v - \frac{1}{n} (1_n^T v) 1_n \qquad (v \in {\mathbb R}^n).
\]
Hence
\[
 a_M = \frac{1}{n} \left( M 1_n - \frac{1}{n} n (1_n^T M 1_n) 1_n \right),
\]
and
\[
b_M = \frac{1}{n} n \left( M^T 1_n - \frac{1}{n} n (1_n^T M^T 1_n) 1_n \right). \qedhere
\]
\end{proof}

\begin{example}
For the Wilson matrix $W$ this decomposition takes the form
\begin{equation}
16W = 1_4
\left (\begin{array}{c}
-27\\
9\\
13\\
5
\end{array}\right )^T
+
\left (\begin{array}{c}
-27\\
9\\
13\\
5
\end{array}\right )
1_4^T
 +\left(
\begin{array}{cccc}
 15 & 11 & -9 & -17 \\
 11 & 23 & -13 & -21 \\
 -9 & -13 & 15 & 7 \\
 -17 & -21 & 7 & 31 \\
\end{array}
\right)
+119\emx_4
\label{eq:sa2}
\end{equation}
and for the integer matrix factor $Z$ of eq.\ (\ref{eq:wilroot})
\begin{equation}
16Z =
1_4
\left (\begin{array}{c}
17\\
1\\
-7\\
-11
\end{array}\right )^T
+
\left (\begin{array}{c}
-7\\
-3\\
5\\
5
\end{array}\right )
1_4^T
 +
\left(
\begin{array}{cccc}
 3 & 15 & -9 & -9 \\
 3 & -1 & 7 & -9 \\
 -5 & -9 & -1 & 15 \\
 -1 & -5 & 3 & 3 \\
\end{array}
\right)+19\emx_4.
\label{eq:sa3}
\end{equation}
The last terms on the right-hand side of these formulae correctly yield
$\mathop{\rm wt} W = \frac{119}{16}$ and
$\mathop{\rm wt} Z = \frac{19}{16}$.
\end{example}

\section{\large Integer factorisation of matrices and the quadratic form obstruction}
\label{sec:quad}

On the basis of the matrix decomposition established in the preceding
section, we now derive the quadratic equation arising from
balancing the weights in a matrix factorisation.

\begin{theorem}
\label{lem:22}
Let $M, N\in\mathbb{R}^{n\times n}$, and let $M = M_V + M_0 + (\mathop{\rm wt} M)\,\emx_n$
and $N = N_V + N_0 + (\mathop{\rm wt} N)\,\emx_n$ be their decompositions as in Theorem
\ref{thm:215}, with $N_V = a 1_n^T + 1_n b^T$.

(i) If $M=N^TN$ then
\begin{equation}
 \mathop{\rm wt} M = |a|^2 + n\,(\mathop{\rm wt} N)^2,
\label{eq:wtqf1}
\end{equation}
$M_V = y 1_n^T + 1_n y^T$ with $y = N_0^T a + n\,(\mathop{\rm wt} N) b$, and
$M_0 = N_0^T N_0+n\,b b^T$.

(ii) If $M = N^2$ then
\begin{equation}
\wt M = b^T a + n\,(\wt N)^2,
\label{eq:wtqf2}
\end{equation}
$M_V = y 1_n^T+1_n z^T$ with $y = N_0^T a + n (\wt N) a$ and $z = N_0^T b + (\wt N) b$, and
$M_0 = N_0^2 + n\,a b^T$.
\end{theorem}

\begin{proof}
(i)
Since $N_0\emx_n=N^T_0\emx_n=\omx_n$, we have
\begin{align*}
N^T N = &(N_V^T N_V + N_0^T N_0)
\\
&+ (N_V^T N_0 + N_0^T N_V + (\wt N) N_V^T\emx_n + (\wt N)\emx_n N_V) + n (\wt N)^2\emx_n,
\end{align*}
where, by the superalgebra property, the matrices in the first bracket lie in $S_n$, the matrices in the second bracket in $V_n$. Writing $N_V = a 1_n^T + 1_n b^T$, we find
$N_V^TN_V=a^T a\emx_n + n b b^T$, $N_V^TN_0=1_na^TN_0$, $N_0^TN_V=N_0^Ta1_n^T$,
$N_V^T\emx_n=n b 1_n^T$ and $\emx_nN_V = n1_n b^T$.
Hence the uniqueness of the decomposition of $M$ by Theorem \ref{thm:215} gives the claimed
identities.

(ii)
Similarly, we find
\begin{align*}
N^2 = &(N_V^2 + N_0^2) + (N_V N_0 + N_0 N_V + (\wt N) N_V \emx_n + (\wt N)\emx_n N_V)
\\
&+ n (\wt N)^2\emx_n,
\end{align*}
and, setting $N_V = a 1_n^T + 1_n b^T$, note that
$N_V^2=b^Ta\emx_n+nab^T$, $N_V N_0 = 1_n b^T N_0$, $N_0 N_V = N_0 a 1_n^T$,
$N_V\emx_n = n a 1_n^T$ and $ \emx_n N_V = n 1_n b^T$. This gives the claimed identities by uniqueness of decomposition.
\end{proof}

The above theorem gives the following quadratic form obstructions to the
factorisation of an integer matrix $M$ into either $M = N^T N$ or $M = N^2$
with an integer matrix $N$.

\begin{corollary}
\label{thm:23}
Given a matrix $M\in\mathbb{Z}^{n\times n}$, it is necessary for the
existence of a factorisation $M = N^T N$ with $N\in\mathbb{Z}^{n\times n}$
that $n^2 \wt N\in\mathbb{Z}$ and that the vector components
$n^2 a_j, n^2 b_j\in\mathbb{Z}$ $(j \in \{1, \dots, n\})$, where
$N = a1_n^T + 1_nb^T + N_0+w_N\emx_n$ is the decomposition of $N$ as in
Theorem \ref{thm:215}, form a solution of the quadratic equation
\begin{align*}
 n^4\wt M &= n\, (n^2 \wt N)^2 + \sum_{j=1}^n (n^2 a_j)^2.
\end{align*}
Similarly, for the decomposition $M = N^2$, the quadratic equation is
\begin{align*}
 n^4\wt M &= n\, (n^2 \wt N)^2 + \sum_{j=1}^n (n^2 a_j)(n^2 b_j).
\end{align*}
\end{corollary}
Suppose we are given a symmetric integer matrix
$M\in\mathbb{Z}^{n\times n}$ and have found a solution
$(w_N, a_1, \dots, a_n) \in \frac 1 {n^2} \mathbb{Z}^{n+1}$ of the
quadratic equation associated with the factorisation $M = N^T N$,
\[
 n^4 \wt M = n^5 w_N^2 + n^4 \sum_{j=1}^n a_j^2.
\]
To complete the factorisation, we need to identify a vector $b$ and a
matrix $N_0$ satisfying the equations in
Theorem \ref{lem:22} (i).
Using the decomposition $M = y 1_n^T + 1_n y^T + M_0 + (\wt M)\,\emx_n$, we have
$b = \frac 1 {n w_N}\,(y - N_0^T a)$, so we only need to find a solution
$N_0$ of the quadratic matrix equation
\be N_0^T (a a^T + n w_N^2\imx_n) N_0 -
N_0^T a y^T - y a^T N_0 = n w_N^2 M_0 - y y^T
\label{eq:qform4}
\ee
or, setting $N_0 = L - \frac 1 {a^T a + n w_N^2}\,a y^T$, the simpler quadratic
\be
 L^T (a a^T + n w_N^2\imx_n) L = n w_N^2 \Bigl(M_0 + \frac 1 {a^T a + n w_N^2} y y^T\Bigr).
\label{eq:qform4a}
\ee
The right-hand side and the middle factor on the left-hand side of these
equations are determined in terms of the $\mathrm{S}+\mathrm{V}$ decomposition of the given
matrix $M$ and the particular solution of the factorisation quadratic form
considered.  Although determining $N_0$ or $L$ from these equations is a
factorisation problem of similar type to the original equation $M = N^T N$,
we found that their solution was computationally more effective.

\begin{rmk}
We remark in passing that the matrix $a a^T + n w_N^2$ appearing as a middle
factor on the left-hand side of eq.\ (\ref{eq:qform4a}) is symmetric positive
definite and can be written as $A^T A$, where
$A = \sqrt{n w_N^2} + \frac{\sqrt{a^Ta + n w_N^2} - \sqrt{n w_N^2}}{a^T a}\, a
a^T$ and
$A^{-1} = \frac 1 {\sqrt{n w_N^2}} + \frac 1 {a^T a}
\left(\frac 1 {\sqrt{a^T a + n w_N)N^2}} - \frac 1 {\sqrt{n w_N^2}} \right) a a^T$.

Thus
$N_0 = A^{-1} L - \frac 1 {a^T a + n w_N^2}\,a y^T$, where
\[
 L^T L = n w_N^2 \Bigl(M_0 + \frac 1 {a^T a + n w_N^2} y y^T\Bigr),
\]
but due to the presence of square roots this is unlikely to give a rational
solution $N_0$.
\end{rmk}

\begin{example}
The quadratic equation arising from balancing the weights in the assumed
factorisation of the Wilson matrix $W=Z^TZ$ is
\begin{align*}
\mathop{\rm wt}W = \frac{119}{16} &=a_1^2+a_2^2+a_3^2+a_4^2 +4w_Z^2
\\
&=2(a_1^2+a_2^2+a_3^2+a_1a_2+a_1a_3+a_2a_3) + 4\,w_Z^2,
\end{align*}
as $a_4=-a_1-a_2-a_3$. Multiplying the equation by $2^7$ and setting $x_i=16a_i$, $w=16w_Z$, we find that a necessary condition for the integer factorisation of the Wilson matrix is that there are integer solutions to the quadratic equation (\ref{eq:qform1})
\[
2 w^2+x_1^2+x_1 x_2+x_1 x_3+x_2^2+x_2 x_3+x_3^2=952.
\]
Solving this equation for
$w,x_1,x_2,x_3$ in Mathematica 11.0 on a PC with a Intel Core i7 6500CPU,
gave the $1728$ solutions in just under 6 seconds.
Exactly one third ($576$) of these solutions lead to
rational matrix factorisations $W=Z^TZ$ with
$Z\in\frac{1}{16}\mathbb{Z}^{4\times 4}$.

The process of converting solutions into matrix factors, i.e. of finding
suitable vectors $b$ and matrices $N_0$ satisfying the equations of Theorem
\ref{lem:22} (i) as outlined after Corollary \ref{thm:23}, took
considerably longer at $34$ minutes. Our approach involved utilising
\eqref{eq:qform4}, in which the vector $b$ is eliminated and the
right-hand-side completely determined for a given factor weight
$w_N$. Substituting potential solutions for the vector $a$ and weight
$w_N$, thus reduces the general problem of finding the matrix $N_0$ to that
of an $(n-1)\times (n-1)$ unknown matrix. For the Wilson matrix this is a
$9$-dimensional problem, and then for each of the $9$-dimensional solutions
for $N_0$, the vector $b$ can be quickly constructed. Adding together the
elements of the decomposition we then recover the rational matrix factors
$N$.

To a large part, the multiplicity of solutions to the factorisation problem is expected.
Indeed, it is clear that if $M = N^T N$ and $U$ is an orthogonal matrix, then $U N$ is
another solution of the factorisation problem; conversely, if $\det M = 1$ and $N$ and $N'$ are solutions with integer entries, then $N' = U N$, where $U = N' N^{-1}$ is an integer orthogonal matrix. It is also not hard to see that any integer orthogonal matrix is a signed permutation matrix, i.e.\ a matrix which has exactly one non-zero entry, either $1$ or $-1$, in each row and in each column.

It is therefore natural to classify the factorisation matrices (integer or
rational) modulo left multiplication with integer orthogonal matrices. For
the factorisations of Wilson's matrix obtained through the above procedure,
this gives three distinct classes, represented by the matrix $Z$ of
(\ref{eq:wilroot}) and the further two matrices \be
\label{eq:Z}
\renewcommand{\arraystretch}{1.35}
Z'=\left(
\begin{array}{cccc}
 \frac{1}{2} & 1 & 0 & 1 \\
 \frac{3}{2} & 2 & 3 & 3 \\
 \frac{1}{2} & 1 & 0 & 0 \\
 \frac{3}{2} & 2 & 1 & 0 \\
\end{array}
\right),
\quad
Z''=\left(
\begin{array}{@{\mskip2mu}rrrr}
 \frac{3}{2} & 2 & 2 & 2 \\
 \frac{3}{2} & 2 & 2 & 1 \\
 \frac{1}{2} & 1 & 1 & 2 \\
 -\frac{1}{2} & -1 & 1 & 1 \\
\end{array}
\right).
\ee
The three factorisations $W = Z^T Z=Z'^T Z'=Z''^T Z''$ correspond to the solutions $(w,x_1,x_2,x_3) = (19,17,1,-7)$, $(w,x_1,x_2,x_3) = (18,-8,20,-12)$ and $(w,x_1,x_2,x_3) = (19,11,7,-1)$ of eq.\ (\ref{eq:qform1}), respectively. Note that $Z'$ and $Z''$ are not integer matrices, so the equivalence class of $Z$ comprises all integer factorisations of the Wilson matrix.
\end{example}

\section{\large Determinants, Decompositions and Adjugates}\label{sec.determ-decomp-adjug}

In this section we consider how the matrix decomposition of Theorem \ref{thm:215} can be used to obtain expressions for the determinant of a matrix.
We begin by recalling that the {\it adjugate\/} of an $n\times n$ matrix $M$ is defined by
\[
\adj M  =((-1)^{i+j}\det(M_{j,i}))_{i,j=1}^n,
\]
where $M_{j,i}$ denotes the submatrix of $M$ obtained by deleting row $j$ and column $i$;
$\adj M$ is the transpose of the matrix of cofactors of $M$. If $M$ is nonsingular, then the adjugate can be written as
$\adj M  =(\det M) M^{-1}$.

The adjugate of a matrix appears naturally when calculating the determinant of a rank-1 perturbation of a matrix, as shown in the following lemma (cf. \cite{higham15}).

\begin{lemma}
\label{lem:41}
Let $M\in\mathbb{C}^{n\times n}$ and $u,v\in\mathbb{C}^n$. Then
\[
\det\left(M + uv^T\right) = \det\left(M\right) + v^T{\rm adj}\left(M\right)u.
\]
\end{lemma}

\begin{proof}
Suppose $M$ is a nonsingular $n\times n$ matrix. Then
\begin{align*}
\det(M+uv^{T}) & =\det(M)\det(I+M^{-1}uv^{T})
=\det(M)(1+v^{T}M^{-1}u)
\\
&=\det(M)+v^{T}(\det(M)M^{-1})u.
\end{align*}
This gives the stated identity for regular matrices. The general case follows from the facts that the set of regular matrices is dense in $\mathbb{C}^{n\times n}$ and that the determinant and adjugate are continuous functions of the matrix.
\end{proof}

\begin{corollary}
\label{lem:42}
Let $M_0\in S_n$ with weight 0, and $M=M_0+w_M\emx_n$ with some $w_M\in\mathbb{C}$. Then
\[
\det M =w_M\,1_n^T ({\rm adj} M_0)1_n = n^2 (\wt M) (\wt \adj M_0).
\]
\end{corollary}

\begin{proof}
Writing $\emx_n=1_n1_n^T$, we have $M=M_0+w_M 1_n1_n^T$, and applying
Lemma~\ref{lem:41} gives
$\det M =\det M_0 +w_m 1_n^T ({\rm adj} M_0)1_n$.  As $M_0 1_n = 0$,
${\det M_0 = 0}$.
\end{proof}

\begin{rmk}
If $M = a 1_n^T + 1_n b^T + M_0 + (\wt M) \emx_n$ is the decomposition of $M\in\mathbb{C}^{n\times n}$ as in Theorem \ref{thm:215}, where $a, b \in \{1_n\}^\bot$ and $M_0 \in S_n$ with weight 0, then Lemma \ref{lem:41} and Corollary \ref{lem:42} can be used to derive the formula
\begin{align*}
\det M &= (\wt M -1)\,1_n^T \adj\left(M_0-a b^T\right)1_n
\\
&\qquad+ (b+1_n)^T \adj \left(M_0-a b^T+(\wt M-1)\emx_n\right)(a+1_n).
\end{align*}
\end{rmk}

\begin{example}
For the Wilson matrix we have $\det(W_S)=\frac{357}{8}$, and ${\rm adj}(W_0)=\frac{3}{8}\emx_4$, where
\begin{align*}
W_S = W_0+w_W &=\frac{1}{8}
\left(
\begin{array}{cccc}
 67 & 65 & 55 & 51 \\
 65 & 71 & 53 & 49 \\
 55 & 53 & 67 & 63 \\
 51 & 49 & 63 & 75 \\
\end{array}
\right)
\\
&=\frac{1}{16}\left(
\begin{array}{cccc}
 15 & 11 & -9 & -17 \\
 11 & 23 & -13 & -21 \\
 -9 & -13 & 15 & 7 \\
 -17 & -21 & 7 & 31 \\
\end{array}
\right)
+\frac{119}{16}\emx_4.
\end{align*}
By Corollary \ref{lem:42}, $\det(W_S)= 16 (\wt W) (\wt \adj W_0)$, and as
$\adj W_0=\frac{3}{8}\emx_4$, this equates to
$16\,\frac{119}{16}\,\frac 3 8 = \frac{357}{8}$.
\end{example}

Thus the adjugate of the weightless type S part of the Wilson matrix is a multiple of $\emx_4$.
As we show in the following theorem, the adjugate of a weightless type S matrix is in fact always a scalar multiple of $\emx_n$. Moreover, the converse holds in the sense that a matrix
whose adjugate is a {\it non-zero\/} multiple of $\emx_n$ must be a weightless type S matrix.
Of course any matrix of rank $n-2$ or less has adjugate $0 = 0 \emx_n$.

\begin{theorem}
(a) Let $M\in S_n$ with weight 0. Then there is a constant $w\in \mathbb{C}$ such that
$\adj M = w\,\emx_{n}$.

(b) Let $M\in\mathbb{C}^{n\times n}$ such that $\adj M = w\,\emx_{n}$ with some $w\in \mathbb{C}\setminus \{0\}$. Then $M\in S_n$, $\wt M = 0$ and $\rank M = n-1$.
\end{theorem}

\begin{proof}
(a) Let $j\in\{1, \dots, n\}$ and denote by $m_{1}$, \dots, $m_{n}\in \mathbb{R}^{n-1}$ the columns of the matrix $M$ with the $j$-th row omitted.
Let $k\in\{1, \dots, n-1\}$. Then
\[
m_{k}=-\sum_{l\in\{1,\dots,n\},l\neq k}m_{l},
\]
as the row sums of $M$ vanish; hence
\begin{align*}
\det(m_{1}, &\dots, m_{k-1}, m_{k}, m_{k+2}, \dots, m_{n})
\\
&=\det(m_{1}, \dots, m_{k-1}, -\sum_{l\neq k}m_{l}, m_{k+2}, \dots, m_{n})
\\
&=-\sum_{l\neq k}\det(m_{1}, \dots, m_{k-1}, m_{l}, m_{k+2}, \dots, m_{n})
\\
&=-\det(m_{1}, \dots, m_{k-1}, m_{k+1}, m_{k+2}, \dots, m_{n}),
\end{align*}
noting that all the determinants in the sum vanish except for the term $l=k+1$. This shows that the $(j, k+1)$ entry and the $(j, k)$ entry of $\adj M$ are equal. Since this holds for all $k\in\{1, \dots, n-1\}$ and for all $j\in\{1, \dots, n\}$, it follows that $\adj M$ has constant columns.

Applying the same argument to $M^{T}$ (which also has row sums $0$), we find that $\adj M$ also has constant rows, and hence $\adj M =w\emx_n$ for some $w\in\mathbb{C}$.

(b) Let $j\in\{1,\dots, n\}$ and let $m_1, \dots, m_n$ be as in (a). Since
\begin{equation}
\det(m_{1}, \dots, m_{n-1})=(-1)^{n+j}w\neq 0,
\label{eq:star1}
\end{equation}
the vectors $m_{1}, \dots, m_{n-1}$ form a basis of $\mathbb{C}^{n-1}$, so
\[
m_{n}=\sum_{l=1}^{n-1}\alpha_{l}m_{l}
\]
for suitable $\alpha_{1}, \dots, \alpha_{n-1}\in \mathbb{C}$.
Now for any $k\in\{1, \dots, n-1\}$ we have
\begin{align*}
w &= (-1)^{k+j}\det(m_{1}, \dots, m_{k-1}, m_{k+1}, \dots, m_{n-1}, m_{n})
\\
&= (-1)^{k+j+n-1-k}\det(m_{1}, \dots, m_{k-1}, m_{n}, m_{k+1}, \dots, m_{n-1}).
\end{align*}
Combining this with eq.\ (\ref{eq:star1}), we find that $0=\det(m_{1},\dots, m_{k-1}, m_{k}+m_{n}, m_{k+1}, \dots, m_{n-1})$ . Hence, for suitable $\beta_{l}\in \mathbb{C}(l\in\{1,\ldots, n-1\}\setminus \{k\})$ we have
\[
\sum_{l\neq k}^{n-1}\beta_{l} m_{l} = m_{k} + m_{n} = m_{k}+\sum_{l=1}^{n-1}\alpha_{l}m_{l}
=(1+\alpha_{k})m_{k} + \sum_{l\neq k}^{n-1}\alpha_{l} m_{l}.
\]
As $m_{1}, \dots, m_{n-1}$ are linearly independent, it follows that $1+\alpha_{k}=0$, i.e. that $\alpha_{k}=-1.$

Since this holds for all $k\in\{1, \dots, n-1\}$, we can conclude that ${\sum\limits_{l=1}^{n}m_{l}=0}$, i.e. that the rows of $M$, except for the jth row, add up to $0$. Since $j\in\{1, \dots, n\}$ was arbitrary, this holds in fact for all rows.

Applying the above reasoning to the transpose of $M$, we find that its columns also add up to $0$.
\end{proof}

\section{\large Co-Latin Matrices}\label{sec:colatin}

An $n\times n$ Latin square (or Latin matrix) is an $n\times n$ matrix with entries from
$\{1,\dots, n\}$ such that the entries of each row and of each column are distinct, i.e.\ each
number in $\{1,\dots, n\}$ appears exactly once in each row and in each column. Evidently
any Latin matrix is of type S.
If $L = (\ell_{p, q})_{p, q=1}^n$ is a Latin square and $k \in \{1, \dots, n\}$, we define
$L^{(-1)}(k) = {\{(p, q) \in \{1, \dots, n\}^2 : \ell_{p, q} = k\}}$.

\begin{defn}
\label{def:cola}
A matrix $M = (m_{i, j})_{i, j=1}^n \in \mathbb{R}^{n\times n}$ has the {\it co-Latin\/} property if
\[
\sum_{(p, q) \in L^{(-1)}(k)} m_{p, q} = 0
\]
for all $k \in \{1, \dots, n\}$ and all $n\times n$ Latin squares $L$.
\end{defn}

Clearly the set of $n\times n$ co-Latin matrices forms a subspace of
$\mathbb{R}^{n\times n}$.  The co-Latin property can be considered an
extreme opposite of the weightless type S property.  A matrix in $S_n$ with
weight 0 has the property that its entries in any one row or column add to
0; a co-Latin matrix has the property that any selection of $n$ entries
such that no two selected entries lie in the same row or the same column
add to 0.  We show in the following that these properties are indeed
complementary in the sense of unique decomposability of any given
weightless $n\times n$ matrix into a weightless type S matrix and a
co-Latin matrix; this is a direct consequence of the following theorem,
which identifies co-Latin matrices with $V_n$ in \eqref{eq:vprop}.

\begin{theorem}
\label{thm:cola}
Let $n\in\mathbb{N}$.
The space of all $n\times n$ co-Latin matrices is equal to $V_n$.
\end{theorem}

To prepare the proof of Theorem \ref{thm:cola}, we first show that there
exists a Latin square which has the entries 1 and 2 pairwise on the
diagonally opposite corners of a rectangle. By suitable row and column
permutations, we can assume without loss of generality that this rectangle
is the top left $2 \times 2$ square. The existence of such a Latin square
is not trivial; in fact there is none in dimension 3, as the arrangement
\[
\left (\begin{array}{ccc}
1 & 2 & * \\
 2 & 1 & * \\
  * & * & *
\end{array} \right )
\]
enforces two 3s in the third column and the third row, violating the defining condition of a Latin square.
\begin{lemma}
\label{thm:31}
Let $n \in {\mathbb N} \setminus \{1, 3\}.$ Then there exists an $n \times n$ Latin square $L=(\ell_{i j})$, such that $\ell_{1\, 1} = \ell_{2\, 2} = 1$ and $\ell_{1\, 2} = \ell_{2\, 1} = 2.$
\end{lemma}
\begin{proof}
We consider the cases of even and odd $n$ separately.

(i) {\it Even\/} $n = 2 m$.
Consider the Hankel Latin square
$\tilde L = (\tilde\ell_{i, j})_{i, j=1}^m$, where $\tilde \ell_{i, j} = 1 + ((i+j-2) \mathop{\rm mod} m)$
$(i, j \in \{1, \dots, m\}),$ e.g. for $m = 4$
\[
 \tilde{L} =
\left ( \begin{array}{cccc}
 1 & 2 & 3 & 4 \\
 2 & 3 & 4 & 1 \\
 3 & 4 & 1 & 2 \\
 4 & 1 & 2 & 3
 \end{array} \right ),
\]
and then replace each $(i, j)$ entry in this matrix with the $2 \times 2$ block
\[
L=\left (\begin{array}{cc}
2 \tilde \ell_{i j} - 1 & 2 \tilde \ell_{i j} \\
 2 \tilde \ell_{i j} & 2 \tilde \ell_{i j} - 1
\end{array}\right )
\]
to create the Latin square $L$.

(ii) {\it Odd\/} $n$. Here we start off the first 3 antidiagonals of $L$ thus:
\[
\left (\begin{array}{cccc}
1 & 2 & 3 & \quad \\
 2 & 1 & \quad & \quad \\
 3 & \quad & \quad & \quad
\end{array}\right ),
\]
and then fill up the antidiagonals up to and including the main antidiagonal in the Hankel
Latin manner described above. We then fill the next three antidiagonals with
\[
\begin{array}{ccc}
 3 \ 1 \ 2 \ 1 \ 2 \ &\cdots &\ 1 \ 2 \ 1 \ 2 \ 3,\\
 2 \ 3 \ 3 \ 3 \ &\cdots& \ 3 \ 3 \ 3 \ 1,\\
 1 \ 2 \ 1 \ &\cdots &\ 2 \ 1 \ 2,
\end{array}
\]
respectively, and complete the remaining antidiagonals in the standard Hankel Latin manner.
For example, for $n = 9$ this gives the Latin square
\[
 L =
\left (\begin{array}{ccccccccc}
1 & 2 & 3 & 4 & 5 & 6 & 7 & 8 & 9 \\
  2 & 1 & 4 & 5 & 6 & 7 & 8 & 9 & 3 \\
  3 & 4 & 5 & 6 & 7 & 8 & 9 & 2 & 1 \\
  4 & 5 & 6 & 7 & 8 & 9 & 1 & 3 & 2 \\
  5 & 6 & 7 & 8 & 9 & 3 & 3 & 1 & 4 \\
  6 & 7 & 8 & 9 & 1 & 3 & 2 & 4 & 5 \\
  7 & 8 & 9 & 2 & 3 & 1 & 4 & 5 & 6 \\
  8 & 9 & 1 & 3 & 2 & 4 & 5 & 6 & 7 \\
  9 & 3 & 2 & 1 & 4 & 5 & 6 & 7 & 8
\end{array}\right )
\]
This construction works from $n = 5$ onwards.
\end{proof}

We can now prove the following.
\begin{lemma}
\label{lem:cola}
Let $n \in {\mathbb N},$ and let $M=(m_{i, j})_{i,j=1}^n$ be an $n \times n$ co-Latin square. Then $M\in V_n$.
\end{lemma}

\begin{proof}
The statement is trivial in the case $n=1$.

Now consider $n \in \mathbb{N}\setminus\{1, 3\}$. Let $i, j, k, l \in \{1, \dots, n\}$ such that $i \neq k$ and $j \neq l$.
By suitable permutation of the rows and columns of the Latin square constructed in Theorem \ref{thm:31}, there exists a Latin square $L=(\ell_{p, q})_{p,q=1}^n$
such that $\ell_{i, j} = \ell_{k, l} = 1$ and $\ell_{i, l} = \ell_{k, j} = 2.$
By the co-Latin property, we know that
\[
 \sum_{(p, q) \in L^{(-1)}(1)} m_{p, q} = 0,
\]
so
\begin{equation}
 m_{i, j} + m_{k, l} = - \sum_{(p, q) \in L^{(-1)}(1) \setminus \{(i, j), (k,l)\}} m_{p, q}.
\label{eq:lala1}
\end{equation}
Now consider the matrix $L' = (\ell_{p, q}')_{p,q=1}^n$ which arises from $L$ by keeping all the same entries except that
$\ell_{i, j}' = \ell_{k, l}' = 2$
and
$\ell_{i, l}' = \ell_{k, j}' = 1.$
Then $L'$ is still a Latin square, and by the co-Latin square property we find that
\[
 \sum_{(p, q) \in {L'}^{(-1)}(1)} m_{p, q} = 0,
\]
so
\begin{equation}
 m_{i, l} + m_{k, j} = -\sum_{(p, q) \in {L'}^{(-1)}(1) \setminus \{(i, l), (k, j)\}} m_{p, q}.
\label{eq:lala2}
\end{equation}
Noting that the index sets of the sums in $(\ref{eq:lala1})$ and $(\ref{eq:lala2}),$
and hence the values of these sums, are the same, we conclude that $M \in V_n$.

Finally, to see that the statement holds true in the case $n = 3$, we note that, up to permutations of the symbols $\{1, 2, 3\}$, there are only two different
$3 \times 3$ Latin squares, namely either of the form
\[
\left (\begin{array}{ccc}
A & B & C \\
B & C & A \\
C & A & B
 \end{array}\right ) \qquad \hbox{\rm or}
\qquad
\left (\begin{array}{ccc}
D & E & F \\
F & D & E \\
E & F & D
\end{array}\right ).
\]
To show that a $3 \times 3$ co-Latin square $M$ has the vertex cross sum
property, without loss of generality we can consider the case $i = j = 1$,
$k = l = 2$ (as the other cases can be reduced to this by suitable row and
column permutations).  Then the first of the above Latin matrices shows
that $m_{2, 1} + m_{1, 2} + m_{3, 3} = 0,$ the second Latin matrix shows
that $m_{1, 1} + m_{2, 2} + m_{3, 3} = 0,$ and it follows that
$m_{1, 1} + m_{2, 2} = m_{2, 1} + m_{1, 2}$.
\end{proof}

We can now complete the proof of Theorem \ref{thm:cola}.
\begin{proof}[Proof of Theorem \ref{thm:cola}]
By Lemma \ref{lem:cola}, any $n\times n$ co-Latin matrix is an element of $V_n$.

Conversely, by Theorem \ref{thm:21} any element of $V_n$ is of the form $a 1_n^T + 1_n b^T$, with suitable $a, b \in \{1_n\}^\bot$.
Since a Latin square takes each value in $\{1, \dots, n\}$ exactly once in each row and each column, it is hence straightforward to see that $a 1_n^T$ and $1_n b^T$ are co-Latin squares.
\end{proof}

\begin{example}
Using the $V$ part of our integer factorisation matrix $Z_0$, given by $Z_V$ in (\ref{eq:sa3}), we have
\[
Z_V=
\frac{1}{8}\left(
\begin{array}{cccc}
 5 & 7 & 11 & 11 \\
 -3 & -1 & 3 & 3 \\
 -7 & -5 & -1 & -1 \\
 -9 & -7 & -3 & -3 \\
\end{array}
\right).
\]
It can be easily verified numerically that this type V matrix satisfies all
$4!=24$ Latin selections summing to 0, and so is a co-Latin matrix.
\end{example}

\section{\large Conclusions}

Matrices with integer entries play an important role in many modern
applications of mathematics.
In numerical analysis, for example,
they make convenient test matrices because they are exactly representable.
Families of matrices with bounded integer entries
have recently been termed Bohemian matrices
and various aspects of them have been studied
\cite{chan1}, \cite{corless1}, \cite{higham1}.
The factorisation of integer
matrices in the form $M = N^T N$ with integer $N$ is related to a classical
topic in number theory, but there has not been much work on finding such
factorisations. In this work we have provided some new ideas to determine
under what circumstances integer factorisations exist and to develop an
approach to computing them.  Our results are founded upon the orthogonal
decomposition of the algebra of square matrices into two parts, of which
one part is the subalgebra of constant (row and column) sum matrices, while
we identified the other part as a space of co-Latin square matrices
with symmetries determined by the properties of Latin squares.

\section*{\large Acknowledgements}

The first author was supported by Engineering and Physical
Sciences Research Council grant EP/P020720/1 and the Royal Society.

\end{document}